\documentclass[12pt]{amsart}
\usepackage{}

\usepackage{cancel}
\usepackage{amsmath}
\usepackage{amsfonts}
\usepackage{amssymb}
\usepackage[all]{xy}           

\usepackage{bbding}
\usepackage{txfonts}
\usepackage{amscd}

\usepackage[shortlabels]{enumitem}
\usepackage{ifpdf}
\ifpdf
  \usepackage[colorlinks,final,backref=page,hyperindex]{hyperref}
\else
  \usepackage[colorlinks,final,backref=page,hyperindex,hypertex]{hyperref}
\fi
\usepackage{tikz}
\usepackage[active]{srcltx}

\makeatletter

\newtheorem{thm}{Theorem}[section]
\newtheorem{lem}[thm]{Lemma}

\newtheorem{pro}[thm]{Proposition}
\newtheorem{ex}[thm]{Example}
\newtheorem{rmk}[thm]{Remark}
\newtheorem{defi}[thm]{Definition}

\newcommand {\emptycomment}[1]{}

\newcommand{\lon }{\,\rightarrow\,}
\newcommand{\be }{\begin{equation}}
\newcommand{\ee }{\end{equation}}

\newcommand{\g}{\mathfrak g}
\newcommand{\h}{\mathfrak h}

\newcommand{\huaM}{\mathcal{M}}

\newcommand{\huaC}{{\mathcal{C}}}
\newcommand{\huaD}{\mathcal{D}}

\newcommand{\huaH}{\mathcal{H}}

\newcommand{\huaO}{{\mathcal{O}}}

\newcommand{\frki}{\mathfrak i}

\newcommand{\frkp}{\mathfrak p}

\newcommand{\frkC}{\mathfrak C}
\newcommand{\frkD}{\mathfrak D}

\newcommand{\dM}{\mathrm{d}}

\newcommand{\EXP}{\mathrm{Exp}}

\newcommand{\Courant}[1]{\left\llbracket  #1\right\rrbracket }

\newcommand{\Id}{{\rm{Id}}}

\newcommand{\br}[1]{   [ \cdot,    \cdot  ]   }

\newcommand{\CE}{\mathsf{CE}}

\newcommand{\Hom}{\mathrm{Hom}}

\newcommand{\Der}{\mathrm{Der}}
\newcommand{\NR}{\mathrm{NR}}

\newcommand{\Ad}{\mathrm{Ad}}
\newcommand{\Aut}{\mathrm{Aut}}

\newcommand{\gl}{\mathfrak {gl}}

\newcommand{\ad}{\mathrm{ad}}

\newcommand{\Img}{\mathrm{Im}}

\newcommand{\Sym}{\mathsf{S}}

\begin{document}

\title{Deformations, cohomologies and integrations of relative difference Lie algebras}

\author{Jun Jiang}
\address{Department of Mathematics, Jilin University, Changchun 130012, Jilin, China}
\email{jiangjmath@163.com}

\author{Yunhe Sheng}
\address{Department of Mathematics, Jilin University, Changchun 130012, Jilin, China}
\email{shengyh@jlu.edu.cn}


\begin{abstract}
In this paper, first using the higher derived brackets, we give the controlling algebra of relative difference Lie algebras, which are also called crossed homomorphisms or differential Lie algebras of weight 1 when the action is the adjoint action. Then using Getzler's twisted $L_\infty$-algebra, we define the cohomology of relative difference Lie algebras. In particular, we define the regular cohomology of difference Lie algebras by which infinitesimal deformations of difference Lie algebras are classified. We also define the cohomology of difference Lie algebras with coefficients in  arbitrary representations, and using the second cohomology group to classify abelian extensions of difference Lie algebras. Finally, we show that any relative difference Lie algebra can be integrated to a relative difference Lie group in a functorial way.
\end{abstract}



\maketitle

\tableofcontents

\allowdisplaybreaks


\section{Introduction}

Crossed homomorphisms on groups already appeared in Whitehead's earlier work ~\cite{Whi}  and were later applied to study non-abelian Galois cohomology~\cite{Se}.
The concept of crossed homomorphisms on Lie algebras was introduced in \cite{Lue} in the study of non-abelian extensions of Lie algebras. By differentiation, crossed homomorphisms on Lie groups give rise to crossed homomorphisms on the corresponding Lie algebras \cite{GLS}. Crossed homomorphisms can be applied to study post-Lie algebras \cite{MQ}. Recently in \cite{PSTZ},
the authors   constructed actions of monoidal categories using crossed homomorphisms on Lie algebras  in the study of representations of Cartan type Lie algebras. The authors also studied deformations and cohomologies of crossed homomorphisms, and classified infinitesimal deformations of crossed homomorphisms using the second cohomology group.

In the definition of a crossed homomorphism $D:\g\to \h$, there is an action $\rho:\g\to\Der(\h)$ of a Lie algebra $\g $ on another Lie algebra $\h$. In particular if the action is the adjoint action of a Lie algebra on itself, a crossed homomorphism is exactly a differential operator of weight 1 \cite{GK,LGG}, which is also called a difference operator \cite{Lev,PS1}. Since in this paper, we not only study operators, but also study Lie algebras, actions and operators simultaneously, so we will use the terminology of a relative difference operator instead of a crossed homomorphism, and use the terminology of a relative difference Lie algebra to indicate the quadruple $(\g,\h,\rho,D)$.

A difference operator  can be viewed as a generalization  of a derivation, and a difference Lie algebra  can be viewed as a generalization  of a LieDer pair  introduced in \cite{TFS}, which consists of a Lie algebra and a derivation on it. Note that the deformation theory and the cohomology theory of LieDer pairs were studied in \cite{TFS}, and generalized  to other algebraic structures \cite{Das,DM}.  In \cite{DL,Lod}, the authors also study associative algebras with derivations from the operadic point of view. Therefore, it is natural to extend the study of deformations and cohomologies of relative difference operators and LieDer pairs to the context of relative difference Lie algebras. Another motivation for such a study due to that difference operators are formal inverses of Rota-Baxter operators, while the deformation  and cohomology theories of the latter are of much interest recently \cite{CC,Das20,LST,TBGS,WZ}. There is a general principle for the  deformation theory of an algebraic structure proposed by Deligne, Drinfeld and Kontsevich: on the one hand, for a given
  algebraic structure, there should be a differential graded Lie algebra (or an $L_\infty$-algebra,   called the controlling algebra) whose Maurer-Cartan elements characterize deformations of this object. On the other hand, there
should be a suitable cohomology so that the infinitesimal of a formal deformation can be identified with a cohomology class. In this paper, we use the higher derived bracket \cite{Vo} to construct the controlling algebra of relative difference Lie algebras, and also introduce the cohomology theory for relative difference Lie algebras.  Applications are given to study infinitesimal deformations and abelian extensions of difference Lie algebras.

Globally a Lie algebra can be integrated to a connected and simply connected Lie group, and a derivation can be integrated to an automorphism. So it is natural to study the integration of a relative difference Lie algebra. Note that it was shown in \cite{GLS} that the differentiation of a relative difference Lie group gives rise to a relative difference Lie algebra. We show that one can integrate a relative difference Lie algebra to a relative difference Lie group in a functorial way.

The paper is organized as follows. In Section \ref{sec:rec}, we recall the controlling algebras and cohomologies of LieAct triples and relative difference operators. In Section \ref{sec:MC}, using the higher derived bracket, we construct an $L_\infty$-algebra, whose Maurer-Cartan elements are relative difference Lie algebra structures (Theorem \ref{thmmcli}). Then we introduce the cohomology theory of  relative difference Lie algebras. In Section \ref{sec:dif}, we apply the above general framework for relative difference Lie algebras to introduce the cohomology of difference Lie algebras. First we introduce the regular cohomology of a   difference Lie algebra, and use the second cohomology group to classify infinitesimal deformations (Theorem \ref{thm:inf-cla}). Then we introduce the cohomology of a   difference Lie algebra with coefficients in an arbitrary representation, and use the second cohomology group to classify abelian extensions of  difference Lie algebras (Theorem \ref{extension 2}).  In Section \ref{sec:int}, we show that a relative difference Lie algebra can be integrated to a relative difference Lie group (Theorem \ref{thm:obj}), and a homomorphism between relative difference Lie algebras can be integrated to a homomorphism between the integrated relative difference Lie groups (Theorem \ref{thm:int-functor}).

\section{The controlling algebras and cohomologies of LieAct triples and relative difference operators}\label{sec:rec}

Let $V$ be a vector space. Define the graded vector space
$\oplus_{n=0}^\infty \Hom(\wedge^{n+1}V,V)$
with the degree of elements in $\Hom(\wedge^nV,V)$ being $n-1$. For $f\in \Hom(\wedge^mV,V), g\in \Hom(\wedge^nV,V)$, the  Nijenhuis-Richardson bracket $[\cdot,\cdot]_\NR$ is defined by
$$ [f,g]_{\NR}:=f\circ g- (-1)^{(m-1)(n-1)}g\circ f,$$
with $f\circ g\in \Hom(\wedge^{m+n-1}V,V)$ being defined by
\begin{equation}
(f\circ g)(v_1,\cdots,v_{m+n-1}):=\sum_{\sigma\in S(n,m-1)} (-1)^\sigma f(g(v_{\sigma(1)},\cdots,v_{\sigma(n)}),v_{\sigma(n+1)}, \cdots,v_{\sigma(m+n-1)}),
\label{eq:fgcirc}
\end{equation}
where the sum is over $(n,m-1)$-shuffles. Recall that a permutation $\tau\in S_n$ is called an  $(i,n-i)$-shuffle if $\tau(1)<\cdots <\tau(i)$ and $\tau(i+1)<\cdots <\tau(n)$.  
Then $\big(\oplus_{n=0}^\infty \Hom(\wedge^{n+1}V,V),[\cdot,\cdot]_{\NR}\big)$ is a   graded Lie algebra \cite{NR,NR2}. With this setup, a Lie algebra structure on $V$ is precisely a degree 1 solution $\omega\in \Hom(\wedge^2V,V)$ of the Maurer-Cartan equation
$$ [\omega,\omega ]_{\NR}=0.$$

\subsection{The controlling algebra and cohomologies of LieAct triples}

 In this subsection, we recall the graded Lie algebra whose Maurer-Cartan elements are LieAct triples, which consists of Lie algebras   $\g$ and $\h$, and  an action  of the Lie algebra $\g$ on $\h$. We also give the cohomologies of LieAct triples as byproducts.

Let $\g$ and $\h$ be vector spaces. The elements in $\g$ are denoted by $x_i$ and the elements in $\h$ are denoted by $u_j$. For a multilinear map $\kappa: \wedge^{k}\g\otimes\wedge^{l}\h\lon \g,$ we define $\hat{\kappa}\in\Hom(\wedge^{k+l}(\g\oplus \h), \g\oplus \h)$ by
\begin{equation*}
\hat{\kappa}(x_1+u_1,\cdots,x_{k+l}+u_{k+l})=\sum_{\tau\in S(k,l)}(-1)^{\tau}\Big(\kappa(x_{\tau(1)},\cdots,x_{\tau(k)},u_{\tau(k+1)},\cdots,u_{\tau(k+l)}),0\Big).
\end{equation*}
Similarly, for $\kappa: \wedge^{k}\g\otimes\wedge^{l}\h\lon \h,$ we define $\hat{\kappa}\in\Hom(\wedge^{k+l}(\g\oplus \h), \g\oplus \h)$ by
\begin{equation*}
\hat{\kappa}(x_1+u_1,\cdots,x_{k+l}+u_{k+l})=\sum_{\tau\in S(k,l)}(-1)^{\tau}\Big(0, \kappa(x_{\tau(1)},\cdots,x_{\tau(k)},u_{\tau(k+1)},\cdots,u_{\tau(k+l)})\Big).
\end{equation*}
The linear map $\hat{\kappa}$ is called a lift of $\kappa$.
Denote by $\g^{k,l}=\wedge^{k}\g\otimes\wedge^{l}\h$. Then $\wedge^{n}(\g\oplus \h)\cong\oplus_{k+l=n}\g^{k,l}$ and $\Hom(\wedge^{n}(\g\oplus \h), \g\oplus \h)\cong(\oplus_{k+l=n}\Hom(\g^{k,l},\g))\oplus(\oplus_{k+l=n}\Hom(\g^{k,l},\h))$, where the isomorphism is the lift.
Moreover, using this isomorphism, we have

\begin{pro}\label{pro:M} {\rm(\cite{CC})}
  The graded vector space
\begin{equation}\label{huaM}
\huaM=\oplus_{k=0}^{+\infty}\Big(\Hom(\wedge^{k+1}\g, \g)\oplus(\oplus_{{i+j=k+1}\atop{j\geq 1}}\Hom(\wedge^{i}\g\otimes\wedge^{j}\h, \h))\Big)
\end{equation} is a graded Lie subalgebra of the graded Lie algebra $L=(\oplus_{k=0}^{+\infty}\Hom(\wedge^{k+1}(\g\oplus \h), \g\oplus \h), [\cdot,\cdot]_{\NR})$.
\end{pro}

Let $(\g, [\cdot,\cdot]_{\g})$ and $(\h, [\cdot,\cdot]_{\h})$ be Lie algebras. Denote by $\Der({\g})$ and $\Der({\h})$ the Lie algebras of derivations on $\g$ and $\h$ respectively. A Lie algebra homomorphism $\rho: \g\lon \Der(\h)$ is called an action of $\g$ on $\h$. Sometimes we denote $[\cdot,\cdot]_\g$ and $[\cdot,\cdot]_\h$ by $\pi$ and $\mu$ respectively in the sequel.

\begin{defi}
Let $(\g,[\cdot,\cdot]_\g)$ and $(\h,[\cdot,\cdot]_\h)$ be Lie algebras and $\rho:\g\lon\Der(\h)$ be an action of $(\g, [\cdot,\cdot]_{\g})$ on $(\h, [\cdot,\cdot]_{\h})$. The triple $(\g, \h, \rho)$ is called a {\bf LieAct triple}.
\end{defi}

\begin{thm}\label{thmMC}{\rm(\cite{CC})}
Let $\g$ and $\h$ be vector spaces. Then Maurer-Cartan elements of the graded Lie algebra  $(\huaM, [\cdot,\cdot]_{\NR})$  are LieAct triple structures.
\end{thm}

Let $(\g,\h,\rho)$ be a LieAct triple. By Theorem \ref{thmMC}, $\Pi=\pi+\rho+\mu$ is a Maurer-Cartan element of the graded Lie algebra $(\huaM, [\cdot,\cdot]_{\NR})$, where  the action $\rho:\g\to\Der(\h)$ is identified with   an element in $\Hom(\g\otimes \h,\h)$. By the graded Jacobi identity, $\dM_{\Pi}=[\Pi,\cdot]_{\NR}$ is a graded derivation of the graded Lie algebra $(\huaM, [\cdot,\cdot]_{\NR})$ satisfying $\dM_{\Pi}^{2}=0$. Therefore we have
\begin{thm}
Let $(\g,\h,\rho)$ be a LieAct triple. Then $(\huaM, [\cdot,\cdot]_{\NR}, \dM_{\Pi})$ is a differential graded Lie algebra.
\end{thm}

Let $(\g,\h,\rho)$ be a LieAct triple. For $n\geq 1$, we define the space of $n$-cochains $C^{n}(\g,\h,\rho)$ to be
\begin{equation*}
C^{n}(\g,\h,\rho)=\Hom(\wedge^{n}\g,\g)\oplus\Big(\oplus_{i=1}^{n}\Hom(\wedge^{n-i}\g\otimes\wedge^{i}\h,\h)\Big).
\end{equation*}
Define the coboundary operator $\frkD:C^{n}(\g,\h,\rho)\lon C^{n+1}(\g,\h,\rho)$ by
\begin{equation*}
\frkD f=(-1)^{n-1}[\Pi, f]_{\NR},\quad \forall f\in C^{n}(\g,\h,\rho).
\end{equation*}

Since $\dM_{\Pi}^2=0$, it follows that $\frkD\circ\frkD=0$. Thus we obtain a cochain complex $(\oplus_{n=1}^{+\infty}C^{n}(\g,\h,\rho),\frkD)$.

\begin{defi}
The cohomology of the cochain complex $(\oplus_{n=1}^{+\infty}C^{n}(\g,\h,\rho),\frkD)$ is defined to be the {\bf cohomology of the LieAct triple} $(\g,\h,\rho)$. Denote the $k$-th cohomology group of the cochain complex $(\oplus_{n=1}^{+\infty}C^{n}(\g,\h,\rho),\frkD)$ by $\huaH^{k}(\g,\h,\rho)$.
\end{defi}

For $f=(f_0, \cdots, f_i,\cdots, f_n)\in C^{n}(\g, \h, \rho)$, where $f_0\in\Hom(\wedge^{n}\g,\g), f_i\in\Hom(\wedge^{n-i}\g\otimes\wedge^{i}\h, \h), 1\leq i\leq n$. Write $\frkD f=((\frkD f)_0, \cdots, (\frkD f)_{n+1})$. By direct computation, we obtain
\begin{eqnarray*}
(\frkD f)_0&=&(-1)^{n-1}([\Pi,f]_{\NR})_0=(-1)^{n-1}[\pi, f_0]_{\NR},\\
(\frkD f)_1&=&(-1)^{n-i}([\Pi,f]_{\NR})_1=(-1)^{n-1}([\pi+\rho, f_1]_{\NR}+[\rho, f_0]_{\NR}),\\
(\frkD f)_{i}&=&(-1)^{n-1}([\Pi,f]_{\NR})_i=(-1)^{n-1}([\pi+\rho, f_i]_{\NR}+[\mu, f_{i-1}]_{\NR}), \quad 2\leq i\leq n-1, \\
(\frkD f)_{n}&=&(-1)^{n-1}([\Pi,f]_{\NR})_n=(-1)^{n-1}([\rho, f_n]_{\NR}+[\mu, f_{n-1}]_{\NR}), \\
(\frkD f)_{n+1}&=&(-1)^{n-1}([\Pi,f]_{\NR})_{n+1}=(-1)^{n-1}([\mu, f_n]_{\NR}).
\end{eqnarray*}
In particular, we have
$$
(\frkD f)_0=(-1)^{n-1}[\pi, f_0]_{\NR}=\dM^{\CE}_{\ad}f_{0},
$$
where $\dM^{\CE}_{\ad}:\Hom(\wedge^{n}\g,\g)\to \Hom(\wedge^{n+1}\g,\g)$ is the Chevalley-Eilenberg coboundary operator of the Lie algebra $\g$ with coefficients in the adjoint representation $\ad$.

\subsection{The controlling algebra and cohomologies of   relative difference operators}

In this subsection, we recall the differential graded Lie algebra whose Maurer-Cartan elements are relative difference operators, and the cohomologies of  relative difference operators.
\begin{defi}
Let $(\g,\h,\rho)$ be a LieAct triple. A linear map $D:\g\lon\h$ is called a {\bf relative  difference operator} with respect to the action $\rho$ if
\begin{equation}\label{eq-defich}
D([x, y]_\g)=\rho(x)D(y)-\rho(y)D(x)+[D(x), D(y)]_{\h}, \quad \forall x,y\in\g.
\end{equation}

A {\bf relative difference Lie algebra}, denoted by $(\g,\h,\rho,D)$, consists of a LieAct triple $(\g, \h, \rho)$ and a relative  difference operator $ D$.

A relative difference operator from $\g$ to $\g$ with respect to the adjoint action is called a {\bf  difference operator}. A Lie algebra $\g$ equipped with a difference operator $D$ is called a {\bf difference Lie algebra}, and denoted by $(\g,D)$.
\end{defi}

\begin{rmk}
\begin{itemize}
\item[\rm(i)] If the action $\rho$ of $\g$ on $\h$ is zero, then a relative difference operator from $\g$ to $\h$ is  a Lie algebra homomorphism.
\item[\rm(ii)] If $\h$ is commutative, then a relative difference operator from $\g$ to $\h$ is simply a derivation from $\g$ to $\h$ with respect to the representation $\rho$.
\end{itemize}
\end{rmk}

\begin{rmk}
We clarify alternative terminologies of relative difference operators. A relative difference operator is also called a crossed homomorphism, which was introduced in \cite{Lue} and further studied in \cite{MQ,PSTZ} recently. When the action is the adjoint action, a difference operator is also called a differential operator of weight $1$. See \cite{GK, GSZ,LGG} for more details.
\end{rmk}

The graphs of maps can be used to characterize relative difference operators on Lie algebras.
\begin{pro}{\rm(\cite{PSTZ})}\label{graphch}
Let $(\g, \h, \rho)$ be a LieAct triple. Then a linear map $D: \g\lon \h$ is a relative difference operator if and only if the graph $G(D)=\{(x,D(x))|x\in\g\}$ of $D$ is a Lie subalgebra of the semi-direct product Lie algebra $\g\ltimes_\rho \h$, where the Lie bracket $[\cdot,\cdot]_\ltimes$ is given by
$$
[x+u,y+v]_\ltimes=[x,y]_\g+\rho(x)v-\rho(y)u+[u,v]_\h,\quad \forall x,y\in\g, u,v \in \h.
$$
\end{pro}

\begin{defi}
Let $(\g,\h,\rho,D)$ and $(\g',\h',\rho',D')$ be two relative difference Lie algebras. A homomorphism from $(\g,\h,\rho,D)$ to $(\g',\h',\rho',D')$ consists of a Lie algebra homomorphism $\psi_\g:\g\rightarrow\g'$ and a Lie algebra homomorphism $\psi_\h:\h\lon\h'$ such that
\begin{eqnarray}
\label{hom-rbo1}D'\circ\psi_\g&=&\psi_\h\circ D,\\
\label{hom-rbo2}\psi_\h\rho(x)(u)&=&\rho'(\psi_\g(x))(\psi_\h(u)),\quad \forall x\in\g, u\in \h.
\end{eqnarray}
\end{defi}
In fact, \eqref{hom-rbo2} is equivalent to that $(\psi_\g, \psi_\h)$ is a homomorphism from the Lie algebra $\g\ltimes _\rho\h$ to $\g\ltimes_{\rho'}\h'$.

In \cite{PSTZ}, the authors gave the Maurer-Cartan characterization of relative difference operators. Let $(\g,\h,\rho)$ be a LieAct triple. Consider the graded vector space
$$\huaC^*(\g,\h):=\oplus_{k\geq 1}\Hom(\wedge^{k}\g,\h).$$
Define a linear map $d_{\pi+\rho}:\Hom(\wedge^m\g,\h)\lon\Hom(\wedge^{m+1}\g,\h)$ by $d_{\pi+\rho}f=[\pi+\rho,f]_\NR$ for all $f\in \Hom(\wedge^m\g,\h)$. Define a skew-symmetric bracket operation $\Courant{\cdot,\cdot}: \Hom(\wedge^m\g,\h)\times \Hom(\wedge^n\g,\h)\longrightarrow \Hom(\wedge^{m+n}\g,\h)$ by
\begin{eqnarray}
&&\nonumber\Courant{f_1,f_2}(x_1,\cdots,x_{m+n})\\
\nonumber&=&(-1)^{(m-1)}[\mu,[f_1,f_2]_\NR]_\NR(x_1,\cdots,x_{m+n})\\
\label{graded-Lie}&=&(-1)^{mn+1}\sum_{\sigma\in  S(m,n)}(-1)^{\sigma}[f_1(x_{\sigma(1)},\cdots,x_{\sigma(m)}),f_2(x_{\sigma(m+1)},\cdots,x_{\sigma(m+n)})]_\h
\vspace{-.1cm}
\end{eqnarray}
for all $f_1\in\Hom(\wedge^m\g,\h)$ and $f_2\in\Hom(\wedge^n\g,\h)$.

\begin{pro}\label{crossed-homo-MC}{\rm(\cite{PSTZ})}
  With the above notations, $(\huaC^*(\g,\h),\Courant{\cdot,\cdot},d_{\pi+\rho})$ is a differential graded Lie algebra. Its Maurer-Cartan elements are precisely   relative difference operators from $\g$ to $\h$ with respect to the action $\rho$.
\end{pro}

 Recall the cohomologies of relative difference operators as following. Let $D:\g\lon\h$ be a relative difference operator with respect to the action $\rho$. Define $\rho_D:\g\to\gl(\h)$ by
\begin{equation}\label{eqreptheta}
\rho_D(x)u=\rho(x)u+[D(x), u]_\h, \quad \forall x\in\g, u\in\h.
\end{equation}
It was shown in \cite[Proposition 3.1]{Lue} that $\rho_D$ is a representation of the Lie algebra $\g$ on $\h$.

Denote by $C^{n}(D)=\Hom(\wedge^{n}\g, \h)$. 

\begin{defi}{\rm(\cite{PSTZ})}
The cohomology of the cochain complex $(\oplus_{n=1}^{+\infty}C^n( D),\dM^{\CE}_{\rho_D})$ is called the {\bf cohomology of the relative difference operator} $D$, where $\dM_{\rho_D}^{\CE}:\Hom(\wedge^{k}\g, \h)\lon\Hom(\wedge^{k+1}\g, \h)$ is the Chevalley-Eilenberg coboundary operator of the Lie algebra $(\g, [\cdot,\cdot]_\g)$ with coefficients in the representation $(\h, \rho_D)$. Denote its $n$-th cohomology group by $\huaH^n(D)$.
\end{defi}

Define a linear map $\dM_{{D}}:\Hom(\wedge^{k}\g, \h)\lon\Hom(\wedge^{k+1}\g,\h)$ by
$\dM_{{D}}=d_{\pi+\rho}+\Courant{{D},\cdot}$. Then there is the following relationship.

\begin{pro}\label{pro:danddT}{\rm(\cite{PSTZ})}
 Let $D:\g\longrightarrow\h$ be a relative difference operator. Then we have
 $$
 \dM^{\CE}_{\rho_{D}}f=(-1)^{k-1}\dM_{D}f=(-1)^{k-1}([\pi+\rho,f]_\NR+[\mu,[D,f]_\NR]_\NR),\quad \forall f\in \Hom(\wedge^k\g,\h).
 $$
\end{pro}

\section{Deformations and cohomologies of relative difference Lie algebras}\label{sec:MC}

In this section, we study deformations and cohomologies of relative difference Lie algebras. First we recall the higher derived brackets, which will be used  to construct the $L_{\infty}$-algebra that control deformations of relative difference Lie algebras.

\subsection{$L_{\infty}$-algebras and the higher derived brackets}

Let $V=\oplus_{k\in\mathbb Z}V^k$ be a $\mathbb Z$-graded vector space.
 The   desuspension operator  $s^{-1}$ changes the grading of $V$ according to the rule $(s^{-1}V)^i:=V^{i+1}$. The  degree $-1$ map $s^{-1}:V\lon s^{-1}V$ is defined by sending $v\in V$ to its   copy $s^{-1}v\in s^{-1}V$.

\begin{defi}
An {\bf  $L_\infty$-algebra} is a $\mathbb Z$-graded vector space $\g=\oplus_{k\in\mathbb Z}\g^k$ equipped with a collection $(k\ge 1)$ of linear maps $l_k:\otimes^k\g\lon\g$ of degree $1$ with the property that, for any homogeneous elements $x_1,\cdots,x_n\in \g$, we have
\begin{itemize}\item[\rm(i)]
{\em (graded symmetry)} for every $\sigma\in S_{n}$,
\begin{eqnarray*}
l_n(x_{\sigma(1)},\cdots,x_{\sigma(n)})=\varepsilon(\sigma)l_n(x_1,\cdots,x_n),
\end{eqnarray*}
\item[\rm(ii)] {\em (generalized Jacobi identity)} for all $n\ge 1$,
\begin{eqnarray*}\label{sh-Lie}
\sum_{i=1}^{n}\sum_{\sigma\in  S(i,n-i) }\varepsilon(\sigma)l_{n-i+1}(l_i(x_{\sigma(1)},\cdots,x_{\sigma(i)}),x_{\sigma(i+1)},\cdots,x_{\sigma(n)})=0,
\end{eqnarray*}

\end{itemize}where $\varepsilon(\sigma)=\varepsilon(\sigma;v_1,\cdots,v_n)$ is the   Koszul sign for a permutation $\sigma\in S_n$ and $v_1,\cdots, v_n\in V$.
\end{defi}

\begin{defi}
 A Maurer-Cartan element of an $L_\infty$-algebra $(\g,\{l_k\}_{k=1}^{+\infty})$ is an element $\alpha\in \g^0$ satisfying the Maurer-Cartan equation
\begin{eqnarray}\label{MC-equation}
\sum_{k=1}^{+\infty}\frac{1}{k!}l_k(\alpha,\cdots,\alpha)=0.
\end{eqnarray}
\end{defi}

\begin{rmk}
In general, the Maurer-Cartan equation \eqref{MC-equation} makes sense when the $L_\infty$-algebra is a filtered $L_\infty$-algebra \cite{DR}.  In the following,  the $L_\infty$-algebra under consideration satisfies  $l_k=0$ for $k$ sufficiently big, so  the Maurer-Cartan equation makes sense.
\end{rmk}

Let $\alpha$ be a Maurer-Cartan element. Define $l_{k}^{\alpha}:\otimes^{k}\g\lon\g ~~(k\geq1)$ by
\begin{equation}
l_{k}^{\alpha}(x_1,\cdots,x_k)=\sum_{n=0}^{+\infty}\frac{1}{n!}l_{k+n}(\underbrace{\alpha,\cdots,\alpha}_{n},x_1,\cdots,x_k).
\end{equation}

\begin{thm}\label{twistLin}{\rm(\cite{Get})}
  $(\g,\{l_k^{\alpha}\}_{k=1}^{+\infty})$ is an $L_\infty$-algebra, called the twisted $L_\infty$-algebra.
\end{thm}

Now we recall the V-data, which is a very powerful tool to construct explicit $L_\infty$-algebras.

\begin{defi}
A $V$-data consists of a quadruple $(L,F,P,\Delta)$, where
\begin{itemize}
\item[$\bullet$] $(L,[\cdot,\cdot])$ is a graded Lie algebra,
\item[$\bullet$] $F$ is an abelian graded Lie subalgebra of $(L,[\cdot,\cdot])$,
\item[$\bullet$] $P:L\lon L$ is a projection, that is $P\circ P=P$, whose image is $F$ and kernel is a  graded Lie subalgebra of $(L,[\cdot,\cdot])$,
\item[$\bullet$] $\Delta$ is an element in $  \ker(P)^1$ such that $[\Delta,\Delta]=0$.
\end{itemize}
\end{defi}

\begin{thm}\label{defili}{\rm (\cite{Vo,Fregier-Zambon-1})}
Let $(L,F,P,\Delta)$ be a $V$-data. Then the graded vector space $s^{-1}L\oplus F$  is an $L_\infty$-algebra where nontrivial products, called the higher derived brackets, are given by
\begin{eqnarray*}\label{V-shla-big-algebra}
l_1(s^{-1}f,\theta)&=&(-s^{-1}[\Delta,f],P(f+[\Delta,\theta])),\\
l_2(s^{-1}f,s^{-1}g)&=&(-1)^fs^{-1}[f,g],\\
l_k(s^{-1}f,\theta_1,\cdots,\theta_{k-1})&=&P[\cdots[[f,\theta_1],\theta_2]\cdots,\theta_{k-1}],\quad k\geq 2,\\
l_k(\theta_1,\cdots,\theta_{k-1},\theta_k)&=&P[\cdots[[\Delta,\theta_1],\theta_2]\cdots,\theta_{k}],\quad k\geq 2.
\end{eqnarray*}
Here $\theta,\theta_1,\cdots,\theta_k$ are homogeneous elements of $F$ and $f,g$ are homogeneous elements of $L$. 

Moreover, if $L'$ is a graded Lie subalgebra of $L$ that satisfies $[\Delta,L']\subset L'$, then $s^{-1}L'\oplus F$ is an $L_\infty$-subalgebra of the above $L_\infty$-algebra $(s^{-1}L\oplus F,\{l_k\}_{k=1}^{+\infty})$.
\end{thm}

\subsection{Deformations  of relative difference Lie algebras}

In this subsection, first we use the higher derived brackets to construct an $L_\infty$-algebra whose Maurer-Cartan elements are relative difference Lie algebras. Then using Getzler's twisting method, we obtain the  $L_\infty$-algebra controlling deformations  of relative difference Lie algebras.

\begin{pro}\label{prodefili}
Let $\g$ and $\h$ be vector spaces. We have a $V$-data $(L, F, P, \Delta)$ as follows:
\begin{itemize}
  \item[$\bullet$] the graded Lie algebra $(L, [\cdot,\cdot])$ is given by $(\oplus_{n=0}^{+\infty}\Hom(\wedge^{n+1}(\g\oplus\h),\g\oplus\h), [\cdot,\cdot]_{\NR})$;
  \item[$\bullet$] the abelian graded Lie subalgebra $F$ is given by $\oplus_{n=0}^{+\infty}\Hom(\wedge^{n+1}\g,\h)$;
  \item[$\bullet$] $P:L\lon L$ is the projection onto the subspace $F$;
  \item[$\bullet$] $\Delta=0$.
\end{itemize}
Consequently, we obtain an $L_{\infty}$-algebra $(s^{-1}L\oplus F, \{l_k\}^{+\infty}_{k=1})$, where $l_{i}$ are given by
\begin{eqnarray*}
l_{1}(s^{-1}f, \theta)&=&P(f),\\
l_{2}(s^{-1}f,s^{-1}g)&=&(-1)^{|f|}s^{-1}[f,g]_{\NR},\\
l_{k}(s^{-1}f,\theta_1,\cdots,\theta_{k-1})&=&P[\cdots[f,\theta_{1}]_{\NR},\cdots,\theta_{k-1}]_{\NR},\quad k\geq 2,
\end{eqnarray*}
for homogeneous elements $\theta, \theta_{1}, \cdots, \theta_{k-1}\in F$, homogeneous elements $f,g\in L$ and all the other possible combinations vanish.
\end{pro}
\begin{proof}
 It is obvious that $\ker(P)$ is a Lie subalgebra. Thus $(s^{-1}L\oplus F, \{l_k\}^{+\infty}_{k=1})$ is an $L_{\infty}$-algebra by Theorem \ref{defili}.
\end{proof}
Now we are ready to give the controlling algebra of relative difference Lie algebras. Recall the graded vector space $\huaM$ given in Proposition \ref{pro:M}:
$$\huaM= \oplus_{n=0}^{+\infty}\Hom(\wedge^{n+1}\g,\g)\oplus\Big(\oplus_{i=1}^{n+1}\Hom(\wedge^{n+1-i}\g\otimes\wedge^{i}\h,\h)\Big).$$

\begin{thm}\label{thmmcli}
Let $\g$ and $\h$ be vector spaces. Then $(s^{-1}\huaM\oplus F, \{l_k\}^{+\infty}_{k=1})$ is an $L_\infty$-algebra, which is a subalgebra of the $L_{\infty}$-algebra $(s^{-1}L\oplus F, \{l_k\}^{+\infty}_{k=1})$.

Moreover, for linear maps $\pi\in\Hom(\wedge^2\g,\g), \mu\in\Hom(\wedge^2\h,\h), \rho\in\Hom(\g\otimes\h,\h)$ and $D\in\Hom(\g,\h)$,    $(s^{-1}(\pi+\mu+\rho), D)$ is a Maurer-Cartan element of the $L_\infty$-algebra $s^{-1}\huaM\oplus F$ if and only if $(\pi+\mu+\rho, D)$ is a relative difference Lie algebra structure.
\end{thm}
\begin{proof}
By Proposition \ref{pro:M}, $\huaM$ is a graded Lie subalgebra. Since $\Delta=0,$ it follows that $[\Delta, \huaM]=0$. Thus $(s^{-1}\huaM\oplus F, \{l_k\}^{+\infty}_{k=1})$ is an $L_{\infty}$-algebra by Theorem  \ref{defili}.

It is straightforward to deduce that
\begin{equation*}
[\pi+\rho+\mu, D]_{\NR}\in\Hom(\wedge^2\g, \h)\oplus\Hom(\g\otimes\h, \h),\quad [[\pi+\rho+\mu, D]_{\NR}, D]_{\NR}\in\Hom(\wedge^2\g,\h),
\end{equation*}
and
\begin{equation*}
[[[\pi+\rho+\mu, D]_{\NR}, D]_{\NR}, D]_{\NR}=0.
\end{equation*}
By Proposition \ref{prodefili}, we have
\begin{eqnarray*}
&&\sum_{k=1}^{+\infty}\frac{1}{k!}l_k((s^{-1}(\pi+\rho+\mu),D),\cdots,(s^{-1}(\pi+\rho+\mu),D))\\
&=&\frac{1}{2}l_2(s^{-1}(\pi+\rho+\mu),s^{-1}(\pi+\rho+\mu))+l_{2}(s^{-1}(\pi+\rho+\mu), D)+\frac{1}{2}l_3(s^{-1}(\pi+\rho+\mu),D,D)\\
&=&(\frac{1}{2}s^{-1}[\pi+\rho+\mu,\pi+\rho+\mu]_{\NR},P([\pi+\rho+\mu,D]_{\NR})+\frac{1}{2}[[\pi+\rho+\mu,D]_{\NR},D]_{\NR}).
\end{eqnarray*}
Thus, $(s^{-1}(\pi+\mu+\rho),D)$ is a Maurer-Cartan element if and only if
\begin{eqnarray}
\label{eq:m1}[\pi+\rho+\mu,\pi+\rho+\mu]_{\NR}&=&0,\\
\label{eq:m2} P([\pi+\rho+\mu,D]_{\NR})+\frac{1}{2}[[\pi+\rho+\mu,D]_{\NR},D]_{\NR}&=&0.
\end{eqnarray}
By Theorem \ref{thmMC}, \eqref{eq:m1} is equivalent to that $\pi+\rho+\mu$ is a LieAct triple structure. Moreover, we have
\begin{eqnarray*}
0&=&P([\pi+\rho+\mu,D]_{\NR})(x_1, x_2)+\frac{1}{2}[[\pi+\rho+\mu,D]_{\NR},D]_{\NR}(x_1, x_2)\\
&=&[\pi+\rho, D]_{\NR}(x_1, x_2)+\frac{1}{2}[[\mu, D]_{\NR}, D]_{\NR}(x_1,x_2)\\
&=&-\rho(x_2)D(x_1)+\rho(x_1)D(x_2)-D(\pi(x_1, x_2))+\mu(D(x_1), D(x_2)).
\end{eqnarray*}
Thus, \eqref{eq:m2} is equivalent to that $D$ is a difference operator. Therefore, $(s^{-1}(\pi+\mu+\rho), D)$ is a Maurer-Cartan element of the $L_\infty$-algebra $s^{-1}\huaM\oplus F$ if and only if $(\pi+\mu+\rho, D)$ is a relative difference Lie algebra structure.
\end{proof}

 By Theorem \ref{twistLin}, we obtain the $L_\infty$-algebra that controls deformations  of relative difference Lie algebras.

 \begin{thm}
   Let $(\g,\h,\rho,D)$ be a relative difference Lie algebra. Then   $(s^{-1}\huaM\oplus F, \{l_{k}^{(s^{-1}\Pi,D)}\}_{k=1}^{+\infty})$ is an $L_\infty$-algebra, where $\Pi=\pi+\rho+\mu$. Moreover, let   $\pi'\in\Hom(\wedge^2\g,\g), \mu'\in\Hom(\wedge^2\h,\h), \rho'\in\Hom(\g\otimes\h,\h)$ and $D'\in\Hom(\g,\h)$ be linear maps. Then $\pi+\pi'$ and $\mu+\mu'$ are  Lie algebra structures on $\g$ and $\h$ respectively, $\rho+\rho'$ is   an action of the Lie algebra $(\g,\pi+\pi')$ on the Lie algebra $(\h,\mu+\mu')$, and $D+D'$ is a relative difference operator if and only if  $(s^{-1}(\pi'+\mu'+\rho'), D')$ is a Maurer-Cartan element of the $L_\infty$-algebra   $(s^{-1}\huaM\oplus F, \{l_{k}^{(s^{-1}\Pi,D)}\}_{k=1}^{+\infty})$.

 \end{thm}

 \begin{proof}
   The first conclusion follows from Theorem \ref{twistLin} directly.

  By Theorem \ref{thmmcli},  the quadruple $(\pi+\pi',\mu+\mu',\rho+\rho',D+D')$ is still a relative difference Lie algebra structure  if and only if
  \begin{eqnarray*}
    \sum_{k=1}^{+\infty}\frac{1}{k!}l_k((s^{-1}(\pi+\pi'+\rho+\rho'+\mu+\mu'),D+D'),\cdots,(s^{-1}(\pi+\pi'+\rho+\rho'+\mu+\mu'),D+D'))=0,
  \end{eqnarray*}
 which is equivalent  to
   \begin{eqnarray*}
    \sum_{k=1}^{+\infty}\frac{1}{k!}l_k^{(s^{-1}\Pi,D)}((s^{-1}(\pi'+\rho'+\mu'),D'),\cdots,(s^{-1}(\pi'+\rho'+\mu'),D'))=0,
  \end{eqnarray*}
  i.e.  $(s^{-1}(\pi'+\mu'+\rho'), D')$ is a Maurer-Cartan element of the $L_\infty$-algebra   $(s^{-1}\huaM\oplus F, \{l_{k}^{(s^{-1}\Pi,D)}\}_{k=1}^{+\infty})$.
 \end{proof}

\subsection{Cohomologies of relative difference Lie algebras}
Let $(\g,\h,\rho,D)$ be a relative difference Lie algebra.  Define the space of $1$-cochains $C^{1}(\g,\h, \rho, D)$ to be $\Hom(\g,\g)\oplus\Hom(\h, \h)$. For $n\geq 2$, define the space of $n$-cochains $C^{n}(\g, \h, \rho, D)$ by
\begin{equation}\label{eq:cochain}
C^{n}(\g,\h, \rho, D)=\Big(\Hom(\wedge^{n}\g,\g)\oplus(\oplus_{i=1}^{n}\Hom(\wedge^{n-i}\g\otimes\wedge^{i}\h,\h))\Big)\oplus\Hom(\wedge^{n-1}\g,\h).
\end{equation}

Define the coboundary operator $\delta: C^{n}(\g,\h, \rho, D)\lon C^{n+1}(\g,\h, \rho, D)$ by
\begin{equation}\label{deficohfrkd}
\delta(f,\theta)=(-1)^{n-2}l_{1}^{(s^{-1}\Pi,D)}(s^{-1}f,\theta).
\end{equation}
where $f\in\Hom(\wedge^{n}\g,\g)\oplus(\oplus_{i=1}^{n}\Hom(\wedge^{n-i}\g\otimes\wedge^{i}\h,\h))$ and $\theta\in\Hom(\wedge^{n-1}\g,\h)$.

\begin{thm}
  $(\oplus_{n=1}^{+\infty}C^{n}(\g,\h, \rho, D), \delta)$ is a cochain complex, i.e. $\delta\circ\delta=0$.
\end{thm}
\begin{proof}
Since $(s^{-1}\huaM\oplus F, \{l_{k}^{(s^{-1}\Pi, D)}\}_{k=1}^{+\infty})$ is an $L_{\infty}$-algebra, we have $l_{1}^{(s^{-1}\Pi, D)}\circ l_{1}^{(s^{-1}\Pi, D)}=0$, which implies that $\delta\circ\delta=0$.
\end{proof}

\begin{defi}
The cohomology of the cochain complex $(\oplus_{n=1}^{+\infty}C^{n}(\g, \h, \rho, D), \delta)$ is called the {\bf cohomology of the relative difference Lie algebra} $(\g, \h, \rho, D)$. We denote its $n$-th cohomology group by $\huaH^{n}(\g, \h, \rho, D)$.
\end{defi}

In the sequel we give the explicit formula of the coboundary operator $\delta$.

For $n\geq 1$, $f=(f_0,\cdots,f_i,\cdots,f_n)\in\Hom(\wedge^{n}\g,\g)\oplus(\oplus_{i=1}^{n}\Hom(\wedge^{n-i}\g\otimes\wedge^{i}\h,\h))$ and $\theta\in\Hom(\wedge^{n-1}\g,\h)$, by Proposition \ref{pro:danddT}, we have
\begin{eqnarray*}
\delta(f, \theta)&=&(-1)^{n-2}l_1^{(s^{-1}\Pi, D)}(s^{-1}f, \theta)\\
&=&(-1)^{n-2}\sum_{k=0}^{+\infty}l_{1+k}(\underbrace{(s^{-1}\Pi,D),\cdots,(s^{-1}(\Pi),D)}_{k}, (s^{-1}f, \theta))\\
&=&(-1)^{n-2}l_2((s^{-1}\Pi, D),(s^{-1}f,\theta))\\
&&+(-1)^{n-2}\sum_{k=2}^{+\infty}\frac{1}{k!}l_{1+k}(\underbrace{(s^{-1}\Pi,D),\cdots,(s^{-1}(\Pi),D)}_{k}, (s^{-1}f, \theta))\\
&=&(-1)^{n-2}(l_2(s^{-1}\Pi, s^{-1}f), l_2(s^{-1}\Pi, \theta)+l_2(s^{-1}f, D))+(-1)^{n-2}l_3(s^{-1}\Pi, D, \theta)\\
&&+(-1)^{n-2}\sum_{k=2}^{+\infty}\frac{1}{k!}l_{1+k}(s^{-1}f, \underbrace{D,\cdots,D}_k)\\
&=&(-1)^{n-2}\Big(-s^{-1}[\Pi, f]_{\NR}, P([\Pi, \theta]_{\NR}+[f, D]_{\NR})+P([[\Pi,D]_{\NR},\theta]_{\NR})\\
&&+\sum_{k=2}^{+\infty}\frac{1}{k!}P(\underbrace{[\cdots[ [}_k f,D]_{\NR},D]_{\NR}\cdots D]_{\NR})\Big)\\
&=&(-1)^{n-2}\Big(-s^{-1}[\Pi, f]_{\NR}, P([\Pi, \theta]_{\NR}+[f, D]_{\NR})+P([[\Pi,D]_{\NR},\theta]_{\NR})\\
&&+\sum_{k=2}^{n}\frac{1}{k!}\underbrace{[\cdots[ [}_k f_{k},D]_{\NR},D]_{\NR}\cdots D]_{\NR}\Big)\\
&=&(-1)^{n-2}\Big(-s^{-1}[\Pi, f]_{\NR}, [\pi+\rho, \theta]_{\NR}+[[\mu,D]_{\NR},\theta]_{\NR}\\
&&+[f_0, D]_{\NR}+[f_1, D]_{\NR}+\sum_{k=2}^{n}\frac{1}{k!}\underbrace{[\cdots[ [}_k f_{k},D]_{\NR},D]_{\NR}\cdots D]_{\NR}\Big)\\
&=&(\frkD f, \dM^{\CE}_{\rho_D}(\theta)+T(f)),
\end{eqnarray*}
where $T:\Hom(\wedge^{n}\g,\g)\oplus(\oplus_{i=1}^{n}\Hom(\wedge^{n-i}\g\otimes\wedge^{i}\h,\h)\lon\Hom(\wedge^n\g,\h)$ is defined by
\begin{eqnarray*}
T(f)=(-1)^{n}\Big(-D\circ f_0+\sum_{k=1}^{n}\frac{1}{k!}\underbrace{[\cdots[ [}_k f_{k},D]_{\NR},D]_{\NR}\cdots D]_{\NR}\Big).
\end{eqnarray*}

\begin{lem}
With the above notions, for all $f=(f_0,\cdots,f_n)\in\Hom(\wedge^{n}\g,\g)\oplus(\oplus_{i=1}^{n}\Hom(\wedge^{n-i}\g\otimes\wedge^{i}\h,\h))$, we have
\begin{eqnarray}\label{lemmadefiT}
&&T(f)(x_1,\cdots,x_n)\\
\nonumber&=&(-1)^{n}\Big(\sum_{k=1}^{n}\sum_{\sigma\in S_{(k,n-k)}}\varepsilon(\sigma)f_k(x_1, \cdots, \hat{x_{i_1}},\cdots, \hat{x_{i_k}},\cdots,x_n ,D(x_{i_1}),\cdots,D(x_{i_k}))\\
\nonumber&&-D(f_0(x_1,\cdots,x_n))\Big).
\end{eqnarray}
\end{lem}
\begin{proof}
 It is well known that   elements in $\oplus_{n=1}^{+\infty}C^{n}(\g\oplus \h;\g\oplus \h)$ correspond to coderivations of the coalgebra  $\bar{\Sym}^c\big(s^{-1}(\g\oplus \h)\big)$. The coderivations corresponding to $f_k$ and $D$ will be denoted by $\bar{f_k}$ and $\bar{D}$ respectively.
 Then, by induction, we have
\begin{eqnarray*}
&&\underbrace{[\cdots[[}_kf_k,D]_{\NR},D]_{\NR},\cdots,D]_{\NR}\big((x_1,v_1),\cdots,(x_{n},v_{n})\big)\\
&=&\sum_{i=0}^{k}(-1)^{i}{k\choose i}\big(\underbrace{\bar{D}\circ\cdots\circ\bar{D}}_i\circ \bar{f_k}\circ \underbrace{\bar{D}\cdots\circ \bar{D}}_{k-i}\big)\big((x_1,v_1),\cdots,(x_{n},v_{n})\big)\\
&=&\big(k!\sum_{\sigma\in S(k,n-k)}\varepsilon(\sigma)f_k(x_1, \cdots, \hat{x_{i_1}},\cdots, \hat{x_{i_k}},\cdots,x_n ,D(x_{i_1}),\cdots,D(x_{i_k})),0\big),
\end{eqnarray*}
for $x_i\in\g, v_i\in\h$.
Thus we have
\begin{eqnarray*}
&&T(f)(x_1,\cdots,x_n)\\
&=&(-1)^{n}\Big(\sum_{k=1}^{n}\sum_{\sigma\in S(k,n-k)}\varepsilon(\sigma)f_k(x_1, \cdots, \hat{x_{i_1}},\cdots, \hat{x_{i_k}},\cdots,x_n ,D(x_{i_1}),\cdots,D(x_{i_k}))\\
&&-D(f_0(x_1,\cdots,x_n))\Big),
\end{eqnarray*}
which finishes the proof.
\end{proof}

\begin{rmk}
 A relative difference operator  reduces to a Lie algebra homomorphism if the action is trivial. The cohomology for  relative difference Lie algebras given above reduces to the cohomology for Lie algebra homomorphisms given   in \cite{Fre}.
\end{rmk}


 At the end of this section, we give the relation between various cohomology groups.

\begin{thm}\label{cohomology-exact-CHL}
There is a short exact sequence of the  cochain complexes:
$$
0\longrightarrow(\oplus_{n=1}^{+\infty}C^n( D),\dM^{\CE}_{\rho_D})\stackrel{\iota}{\longrightarrow}(\oplus_{n=1}^{+\infty}C^n(\g,\h,\rho,D), {\delta})\stackrel{p}{\longrightarrow} (\oplus_{n=1}^{+\infty}C^n (\g ,\h,\rho),\frkD)\longrightarrow 0,
$$
where $\iota(\theta)=(0,\theta)$ and $p(f,\theta)=f$ for all $f\in C^n(\g,\h,\rho)$ and $\theta\in \Hom(\wedge^{n-1}\g,\h)$.

Consequently,
there is a long exact sequence of the  cohomology groups:
$$
\cdots\longrightarrow\huaH^n(D)\stackrel{\huaH^n(\iota)}{\longrightarrow}\huaH^n(\g,\h,\rho,D)\stackrel{\huaH^n(p)}{\longrightarrow} \huaH^n(\g,\h,\rho)\stackrel{c^n}\longrightarrow \huaH^{n+1}(D)\longrightarrow\cdots,
$$
where the connecting map $c^n$ is defined by
$c^n([\alpha])=[T(\alpha)],$  for all $[\alpha]\in \huaH^n(\g,\h,\rho).$
\end{thm}
\begin{proof}
 By the explicit formula of the coboundary operator $\delta$, we have the short exact sequence  of chain complexes which induces a long exact sequence of cohomology groups.
\end{proof}

\section{Cohomologies of difference Lie algebras and applications}\label{sec:dif}

In this section, we use the above general framework for relative difference Lie algebras to give the cohomology theory for difference Lie algebras. First we introduce the regular cohomology of a difference Lie algebra and classify infinitesimal deformations using the second cohomology group. Then we introduce the cohomology of a difference Lie algebra with coefficients in an arbitrary representation. As applications, we classify abelian extensions using the second cohomology group.

\subsection{Regular cohomologies of difference Lie algebras and infinitesimal deformations}\label{cohomadj}

  Let $(\g, D)$ be a difference Lie algebra. Define the space of $1$-cochains $C^{1}(\g,D)$ to be $\Hom(\g, \g)$. For $n\geq 2$, define the space of $n$-cochains $C^{n}(\g,D)$ by
 $$C^{n}(\g,D)=\Hom(\wedge^n\g,\g)\oplus\Hom(\wedge^{n-1}\g,\g).$$

Define the embedding $\frki:C^n(\g,D)\lon C^n(\g,\g, \ad, D)$ by
$$
\frki(f,\theta)=(\underbrace{f,\cdots,f}_{n+1},\theta),\quad \forall f\in \Hom(\wedge^n\g,\g), \theta\in\Hom(\wedge^{n-1}\g,\g),
$$
where $C^n(\g,\g, \ad, D)$ is the space of $n$-cochains given by \eqref{eq:cochain} for the relative difference Lie algebra $(\g,\g, \ad, D)$. Denote by $\Img^n(\frki)=\frki(C^n(\g,D))$. Then it is straightforward to deduce that
$(\oplus_{n=1}^{+\infty}\Img^n(\frki),\delta)$ is a subcomplex of the cochain complex $(\oplus _{n=1}^{+\infty}C^n(\g,\g,\ad,D),\delta)$.

Define the projection $\frkp:\Img^n(\frki)\lon C^n(\g,D)$ by
$$
\frkp(\underbrace{f,\cdots,f}_{n+1}, \theta)=(f,\theta), \quad \forall f  \in \Hom(\wedge^n\g,\g), \theta\in\Hom(\wedge^{n-1}\g,\g).
$$
Then for $n\geq 1,$ we define $\bar{\delta}:C^n(\g,D)\lon C^{n+1}(\g,D)$ by
$
\bar{\delta}=\frkp\circ \delta\circ \frki.
$

More precisely,
\begin{equation}\label{eq:Dexplicit}
\bar{\delta}(f,\theta)=(\dM^{\CE}_{\ad}f, \dM^{\CE}_{\ad_D}\theta+T(f)),
\end{equation}
for all $f\in\Hom(\wedge^n\g,\g)$ and $\theta\in\Hom(\wedge^{n-1}\g,\g)$, where $\dM^{\CE}_{\ad}:\Hom(\wedge^n\g, \g)\lon\Hom(\wedge^{n+1}\g,\g)$ and $\dM^{\CE}_{\ad_D}$ are the Chevalley-Eilenberg coboundary operators of the Lie algebra $\g$ with coefficients in the adjoint representation $\ad$ and the representation $\ad_D$ respectively. By \eqref{lemmadefiT}, the linear map $T:\Hom(\wedge^{n}\g,\g)\lon\Hom(\wedge^{n}\g,\g)$ is given by
\begin{eqnarray}\label{defitranT}
&&T(f)(x_1,\cdots,x_{n})\\
\nonumber&=&(-1)^{n}\Big(\sum_{k=1}^{n}\sum_{1\leq i_1<\cdots<i_k\leq n}f(x_1,\cdots,x_{{i_1}-1},D(x_{i_1}),\cdots,D(x_{i_k}),\cdots,x_n)-D(f(x_1,\cdots,x_n))\Big),
\end{eqnarray}
for all $x_i\in\g$.

\begin{thm}\label{1}
With the above notations,  $(\oplus _{n=1}^{+\infty}C^n(\g, D), \bar{\delta})$ is a cochain complex, i.e. $\bar{\delta}\circ\bar{\delta}=0$.
\end{thm}
\begin{proof}
Since $(\oplus_{n=1}^{+\infty}\Img^n(\frki),\delta)$ is a subcomplex of the cochain complex $(\oplus _{n=1}^{+\infty}C^n(\g,\g,\ad,D),\delta)$ and $\frki\circ\frkp=\Id$, we have
  \begin{eqnarray*}
   \bar{\delta}\circ\bar{\delta}=\frkp\circ \delta\circ \frki\circ \frkp\circ \delta\circ \frki=\frkp\circ \delta\circ   \delta\circ \frki=0,
  \end{eqnarray*}
 which finishes the proof.
\end{proof}

\begin{defi}
Let $(\g, D)$ be a difference Lie algebra. The cohomology of the cochain complex $(\oplus_{n=1}^{+\infty}C^{n}(\g, D), \bar{\delta})$ is taken to be the {\bf regular cohomology of the difference Lie algebra} $(\g, D)$. Denote the $n$-th cohomology group by $\huaH^{n}(\g, D)$.
\end{defi}

In the sequel, we use the second regular cohomology group to classify infinitesimal deformations of difference Lie algebras.

Let $(\g, D)$ be a difference Lie algebra over $\mathbb{R}$ and $\mathbb{R}[t]$ be the polynomial ring in one variable $t$. Then $\mathbb{R}[t]/(t^{2})\otimes_{\mathbb{R}}\g$ is an $\mathbb{R}[t]/(t^{2})$-module. Moreover, $\mathbb{R}[t]/(t^{2})\otimes_{\mathbb{R}}\g$ is a difference Lie algebra over $\mathbb{R}[t]/(t^{2})$, where the Lie bracket $[\cdot,\cdot]$ is defined by
\begin{equation*}
[f(t)x, g(t)y]=f(t)g(t)[x,y]_{\g}, \quad \forall f(t),g(t)\in\mathbb{R}[t]/(t^2), x,y\in\g,
\end{equation*}
and the difference operator $D$ is defined by
\begin{equation*}
D(f(t)x)=f(t)D(x), \quad \forall f(t)\in\mathbb{R}[t]/(t^2), x\in\g.
\end{equation*}

\begin{defi}
Let $(\g, D)$ be a difference Lie algebra and $\hat{\omega}:\wedge^{2}\g\lon\g, ~\hat{D}:\g\lon\g$ be linear maps. If $[\cdot,\cdot]_{t}=[\cdot,\cdot]+t\hat{\omega}$ endow $\mathbb{R}[t]/(t^2)\otimes_{\mathbb{R}}\g$ a Lie algebra structure and $D_{t}=D+t\hat{D}$ is still a difference operator on the Lie algebra $(\mathbb{R}[t]/(t^{2})\otimes_{\mathbb{R}}\g,[\cdot,\cdot]_{t})$, we say that $(\hat{\omega}, \hat{D})$ generates an {\bf infinitesimal deformation} of the difference Lie algebra $(\g, D)$.
\end{defi}

\begin{thm}\label{thmdefH2}
If $(\hat{\omega}, \hat{D})$ generates an infinitesimal deformation of the difference Lie algebra $(\g, D)$, then $(\hat{\omega}, \hat{D})$ is a $2$-cocycle, that is $\bar{\delta}(\hat{\omega}, \hat{D})=0$.
\end{thm}
\begin{proof}
 Since $[\cdot,\cdot]_{t}$ is a Lie algebra, then for any $x, y, z\in\g$, we have
\begin{equation*}
[[x, y]_{t}, z]_{t}+[[y,z]_t,x]_t+[[z,x]_t, y]_t=0,
\end{equation*}
which implies that $\dM^{\CE}_{\ad}(\hat{\omega})=0$.

Moreover, since $D_t=D+t\hat{D}$ is a difference operator,  we have
\begin{eqnarray*}
D_{t}[x, y]_t=[D_{t}(x), y]_t+[x, D_{t}(y)]_t+[D_t(x), D_t(y)]_t,
\end{eqnarray*}
which implies that
\begin{eqnarray*}
&&\dM^{\CE}_{\ad_D}\hat{D}(x,y)+T(\hat{\omega})(x,y)\\
&=&[x, \hat{D}(y)]_\g+[D(x), \hat{D}(y)]_\g+[\hat{D}(x),y]_\g+[\hat{D}(x),D(y)]_\g-\hat{D}([x,y]_\g)\\
&&+\hat{\omega}(D(x),y)+\hat{\omega}(x,D(y))+\hat{\omega}(D(x),D(y))-D(\hat{\omega}(x,y))\\
&=&0.
\end{eqnarray*}
Thus we have $\bar{\delta}(\hat{\omega}, \hat{D})=0$.
\end{proof}

\begin{defi}\label{defidefoequ}
Let $(\g, D)$ be a difference Lie algebra. Two infinitesimal deformations $([\cdot,\cdot]_{t}^{1}, D_t^{1})$ and $([\cdot,\cdot]_{t}^{2}, D_{t}^{2})$ generated by $(\hat{\omega}_1, \hat{D}_1)$ and $(\hat{\omega}_2, \hat{D}_2)$ are said to be {\bf equivalent} if there exists a linear map $N\in\Hom(\g, \g)$ such that the $\mathbb{R}[t]/(t^2)$-module map
$\varphi_t=\Id_\g+tN$
satisfies the following conditions:
\begin{itemize}
\item[\rm(i)] $\varphi_{t}([x, y]_{t}^1)=[\varphi_{t}(x),\varphi_{t}(y)]_{t}^2,\quad \forall x,y\in\g,$
\item[\rm(ii)] $D_{t}^2\circ\varphi_{t}=\varphi_{t}\circ D_{t}^1$ as $\mathbb{R}[t]/(t^2)$-module maps.
\end{itemize}
\end{defi}

\begin{thm}\label{thm:inf-cla}
Let $( \g , D)$ be a difference Lie algebra. The equivalence classes of infinitesimal deformations of $(\g, D)$ are in one-to-one correspondence with the second regular cohomology group $\huaH^{2}(\g, D)$.
\end{thm}
\begin{proof}
Let $([\cdot,\cdot]_{t}^{1}, D_{t}^{1})$ and $([\cdot,\cdot]_{t}^{2}, D_{t}^{2})$ be two equivalent infinitesimal deformations generated by $(\hat{\omega}_1, \hat{D}_1)$ and $(\hat{\omega}_2, \hat{D}_2)$. By Theorem \ref{thmdefH2}, we have $[(\hat{\omega}_1, \hat{D}_1)]\in\huaH^{2}(\g, D)$ and $[(\hat{\omega}_2, \hat{D}_2)]\in\huaH^{2}(\g, D)$. Let $\varphi_{t}=\Id_\g+tN$ satisfy (i) and (ii) in Definition \ref{defidefoequ}. Since $\varphi_{t}([x, y]_{t}^{1})=[\varphi_{t}(x),\varphi_{t}(y)]_{t}^{2}$, then
\begin{equation}\label{eqeqdefor}
\hat{\omega}_1(x, y)-\hat{\omega}_2(x,y)=[N(x), y]_\g-[N(y), x]_\g-N([x, y]_\g),
\end{equation}
that is $\hat{\omega}_1-\hat{\omega}_2=\dM_{\ad}^{\CE}N$.

 Moreover, we have
\begin{equation*}
(\Id_\g+tN)(D+t\hat{D}_1)(x)=(D+t\hat{D}_2)(\Id_\g+tN)(x), \quad \forall x\in\g,
\end{equation*}
which implies that
\begin{equation}
\label{equa-defor}\hat{D}_1(x)-\hat{D}_2(x)=D(N(x))-N(D(x)),\quad \forall x\in\g.
\end{equation}

By \eqref{eqeqdefor} and \eqref{equa-defor}, we have
\begin{equation*}
(\hat{\omega}_1, \hat{D}_1)-(\hat{\omega}_2, \hat{D}_2)=\bar{\delta}(N),
\end{equation*}
which implies that $[(\hat{\omega}_1, \hat{D}_1)]=[(\hat{\omega}_2, \hat{D}_2)]$.

Conversely, suppose that $[(\hat{\omega}_1, \hat{D}_1)]\in\huaH^{2}(\g, D)$. Define $[\cdot, \cdot]_{t}^{1}$ and $D_{t}^{1}$ by
\begin{equation*}
[x,y]_{t}^{1}=[x,y]_\g+t\hat{\omega}_1(x,y), \quad D_{t}^{1}(x)=D(x)+t\hat{D}_1(x),\quad \forall x,y \in\g.
\end{equation*}
By $\dM_{\ad}^{\CE}\hat{\omega}_1=0$ and $\dM_{\ad_D}^{\CE}\hat{D}_1+T(\hat{\omega}_1)=0$,   $([\cdot,\cdot]_{t}^{1}, D_{t}^{1})$ is an infinitesimal deformation of $(\g, D)$.

If $[(\hat{\omega}_2, \hat{D}_2)]=[(\hat{\omega}_1, \hat{D}_1)]$, then there is an infinitesimal deformation generated by $(\hat{\omega}_2, \hat{D}_2)$ and there exists a linear map $N:\g\lon\g$ such that $\hat{\omega}_1-\hat{\omega}_2=\dM_{\ad}^{\CE}N$ and $\hat{D}_1-\hat{D}_2=T(N)$. Then we have
\begin{equation*}
(\Id_\g+tN)([x,y]_t^1)=[(\Id_\g+tN)x, (\Id_\g+tN)y]_t^{2},
\end{equation*}
and
\begin{equation*}
(\Id_\g+tN)(D+t\hat{D}_1)(x)=(D+t\hat{D}_2)(\Id_\g+tN)(x), \quad \forall x\in\g,
\end{equation*}
which implies that two infinitesimal deformations $([\cdot,\cdot]_{t}^{1}, D_{t}^{1})$ and $([\cdot,\cdot]_{t}^2, D_t^2)$ are equivalent.
\end{proof}

\subsection{Classification of abelian extensions of difference Lie algebras}

 In this subsection, first we introduce the notion of a representation of a difference Lie algebra and develop a cohomology theory of a difference Lie algebra with coefficients in an arbitrary representation. Finally we classify abelian extensions of difference Lie algebras.
\begin{defi}
A {\bf representation of a difference Lie algebra} $(\g, D)$ on a vector space $V$ with respect to a linear map $K: V\lon V$ is a representation $\varrho:\g\rightarrow\gl(V)$ of the Lie algebra $\g$ on the vector space $V$ such that the following equation is satisfied:
\begin{equation}\label{defirepcrohomor}
K(\varrho(x)u)=\varrho(D(x))u+\varrho(x)K(u)+\varrho(D(x))K(u), \quad\forall x\in\g, u\in V.
\end{equation}
\end{defi}
We denote a representation by $(V,\varrho,K)$.

\begin{ex}
Let $(\g,D)$ be a difference Lie algebra. Then $(\g, \ad, D)$ is a representation, which is called the {\bf adjoint representation of the difference Lie algebra} $(\g, D)$, where $\ad$ is the adjoint representation of the Lie algebra $\g$.
\end{ex}

\begin{pro}\label{semiRB}
Let $(V, \varrho, K)$ be a representation of a difference Lie algebra $(\g, D)$. Then $(\g\oplus V, [\cdot,\cdot]_{\ltimes}, D+K)$ is a difference Lie algebra, where $[\cdot,\cdot]_{\ltimes}$ is the semidirect product Lie bracket given by
\begin{equation*}
[x+u,y+v]_{\ltimes}=[x,y]_\g+\varrho(x)v-\varrho(y)u,\quad \forall x, y\in\g, u,v\in V,
\end{equation*}
and the difference operator $D+K: \g\oplus V\rightarrow\g\oplus V$ is given by
\begin{equation*}
(D+K)(x+u)=D(x)+K(u).
\end{equation*}
This difference Lie algebra is called the {\bf semi-direct product} of $(\g,  D)$ and the representation $(V, \varrho, K)$, and denoted by $(\g\ltimes_{\varrho}V, D+K)$.
\end{pro}
\begin{proof}
Since $\varrho$ is a representation of $\g$ on $V$, it is obvious that  $(\g\oplus V, [\cdot,\cdot]_{\ltimes})$ is a Lie algebra.

By \eqref{defirepcrohomor}, we obtain
\begin{eqnarray*}
&&[(D+K)(x+u), y+v]_{\ltimes}+[x+u, (D+K)(y+v)]_{\ltimes}+[(D+K)(x+u), (D+K)(y+v)]_{\ltimes}\\
&=&[D(x), y]_\g+\varrho(D(x))v-\varrho(y)K(u)+[x, D(y)]_\g+\varrho(x)K(v)-\varrho(D(y))u\\
&&+[D(x), D(y)]_\g+\varrho(D(x))K(v)-\varrho(D(y))K(u)\\
&=&[D(x), y]_\g+[x, D(y)]_\g+[D(x), D(y)]_\g+K(\varrho(x)v)-K(\varrho(y)u)\\
&=&D([x, y]_\g)+K(\varrho(x)v-\varrho(y)u)\\
&=&(D+K)([x+u, y+v]_{\ltimes}),
\end{eqnarray*}
which implies that $D+K$ is a difference operator. Thus $(\g\ltimes_{\varrho}V, D+K)$ is a difference Lie algebra.
\end{proof}

Now we introduce a cohomology theory of difference Lie algebras with coefficients in  arbitrary representations.

Let $(V, \varrho, K)$ be a representation of $(\g, D)$.  Define the space of $1$-cochains $\frkC^1(\g, D; V, \varrho, K)$ to be $\Hom(\g,V)$. For $n\geq 2$, define the space of $n$-cochains $\frkC^n(\g, D; V, \varrho, K)$  by
$$
\frkC^n(\g,D; V, \varrho, K)=\Hom(\wedge^{n}\g,V)\oplus \Hom(\wedge^{n-1}\g ,V).
$$
Define the coboundary operator
\begin{equation*}
\delta_{\varrho}:\frkC^n(\g,D; V, \varrho, K)\lon \frkC^{n+1}(\g,D; V, \varrho, K)
\end{equation*}
by
\begin{equation*}
\delta_{\varrho}(f,\theta)=(\dM^{\CE}_{\varrho} f,\partial\theta+T(f)),
\end{equation*}
for all $f\in\Hom(\wedge^{n}\g,V), \theta\in\Hom(\wedge^{n-1}\g,V)$,
where $\dM^{\CE}_{\varrho}$ is the Chevalley-Eilenberg coboundary operator of the Lie algebra $\g$ with coefficients in the representation $(V,\varrho)$, and
\begin{itemize}
\item $\partial:\Hom(\wedge^{n-1}\g,V)\lon \Hom(\wedge^{n}\g,V)$ is defined by
\begin{eqnarray*}
&&\partial\theta(x_{1},\cdots,x_{n})\\
&=&\sum_{i=1}^{n}(-1)^{i+1}\varrho(x_i)\theta(x_1,\cdots,\hat{x_i},\cdots,x_{n})+\sum_{i=1}^{n}(-1)^{i+1}\varrho(D(x_i))\theta(x_1,\cdots,\hat{x_i},\cdots,x_{n})\\
\nonumber&&+\sum_{i<j}(-1)^{i+j}\theta([x_i,x_j]_{\g},x_1,\cdots,\hat{x_i},\cdots,\hat{x_j},\cdots,x_{n}).
\end{eqnarray*}
\item $T: \Hom(\wedge^{n}\g,V)\rightarrow\Hom(\wedge^{n}\g, V)$ is defined by
\begin{eqnarray*}
&&T(f)(x_1,\cdots,x_{n})\\
\nonumber&=&(-1)^{n}\Big(\sum_{k=1}^{n}\sum_{1\leq i_1<\cdots<i_k\leq n}f(x_1,\cdots,x_{i-1},D(x_{i_1}),\cdots,D(x_{i_k}),\cdots,x_n)-K(f(x_1,\cdots,x_{n}))\Big).
\end{eqnarray*}

 \end{itemize}

\begin{thm}\label{cohomology-of-LLT}
  With the above notations,  $(\oplus _{n=1}^{+\infty}\frkC^n(\g, D; V,\varrho, K),\delta_{\varrho})$ is a cochain complex, i.e. $$\delta_{\varrho}\circ \delta_{\varrho}=0.$$
\end{thm}
\begin{proof}
We only give a sketch of the proof and leave details to readers. Consider the semi-direct product difference Lie algebra $ (\g\ltimes_{\varrho}V, D+K)$ given in Proposition \ref{semiRB}, and the associated cochain complex $(\oplus_{n=1}^{+\infty}C^n(\g\oplus V, D+K),\bar{\delta})$ given in Theorem \ref{1}. It is straightforward to deduce that
  $(\oplus _{n=1}^{+\infty}\frkC^n(\g, D; V, \varrho, K),\delta_{\varrho})$ is a subcomplex of $(\oplus_{n=1}^{+\infty}\frkC^n(\g\oplus\h, D+K),\bar{\delta})$. Thus, $\delta_{\varrho}\circ \delta_{\varrho}=0.$
\end{proof}

\begin{defi}
  The  cohomology of the cochain complex  $(\oplus_{n=1}^{+\infty}\frkC^n(\g,D; V,\varrho,K),\delta_{\varrho})$ is called {\bf the  cohomology of the difference Lie algebra} with coefficients in the representation $(V, \varrho, K)$. The corresponding $n$-th cohomology group is denoted by $\huaH^n(\g, D; V, \varrho, K)$.
\end{defi}

In the sequel, we use the second   cohomology groups of  difference Lie algebras with coefficients in arbitrary representations to classify abelian extensions of difference Lie algebras.

\begin{defi}
Let $(\g,D)$ and $(\h,K)$ be two difference Lie algebras. An {\bf extension} of $(\g,D)$ by $(\h,K)$ is a short exact sequence of difference Lie algebra homomorphisms:
\[\begin{CD}
0@>>>\h@>i>>\hat{\g}@>p>>\g            @>>>0\\
@.    @V K VV   @V\hat{D}VV  @V D VV    @.\\
0@>>>\h @>i>>\hat{\g}@>p>>\g             @>>>0
,
\end{CD}\]
where $(\hat{\g}, \hat{D})$ is a difference Lie algebra.

An extension of $(\g, D)$ by $(\h, K)$ is called {\bf abelian} if $\h$ is an abelian Lie algebra.
\end{defi}
\begin{defi}
A {\bf section} of an extension $(\hat{\g}, \hat{D})$ of a difference Lie algebra $(\g, D)$ by $(\h, K)$ is a linear map $ s:\g\rightarrow\hat{\g}$ such that
\begin{equation*}
p\circ s=\Id.
\end{equation*}
\end{defi}

In the sequel, we only consider abelian extensions. Let $s$ be a section. Define a linear map $\varrho:\g\rightarrow\gl(\h)$ by
\begin{equation*}
\varrho(x)u=[s(x),u]_{\hat{\g}},\quad \forall x\in\g, u\in\h.
\end{equation*}
Then we have the following result.

\begin{pro}\label{without section}
With the above notations, the linear map $\varrho:\g\lon\gl(\h)$ is a representation of the difference Lie algebra $(\g, D)$ on the vector space $\h$ with respect to the linear map $K$.
Moreover, this representation is independent on the choice of sections.
\end{pro}
\begin{proof}
For any $x,y\in\g, u\in\h$, since $\h$ is abelian, we have
\begin{eqnarray*}
\nonumber\varrho([x,y]_\g)u&=&[s([x,y]_\g),u]_{\hat{\g}}\\
\nonumber&=&[[s(x),s(y)]_{\hat{\g}}+s([x,y]_\g)-[s(x),s(y)]_{\hat{\g}},u]_{\hat{\g}}\\
\nonumber&=&[[s(x),s(y)]_{\hat{\g}},u]_{\hat{\g}}\\
\nonumber&=&[[s(x),u]_{\hat{\g}},s(y)]_{\hat{\g}}+[s(x),[s(y),u]_{\hat{\g}}]_{\hat{\g}}\\
\label{rep2}&=&[\varrho(x),\varrho(y)]u.
\end{eqnarray*}
Thus $\varrho$ is a representation of the Lie algebra $(\g, [\cdot,\cdot]_\g)$ on the vector space $\h$.

Furthermore, since $s(D(x))-\hat{D}(s(x))\in\h$ and $\h$ is abelian, then we have
\begin{eqnarray*}
K(\varrho(x)u)&=&K([s(x),u]_{\hat{\g}})=\hat{D}([s(x),u]_{\hat{\g}})\\
&=&[\hat{D}(s(x)),u]_{\hat{\g}}+[s(x), \hat{D}(u)]_{\hat{\g}}+[\hat{D}(s(x)), \hat{D}(u)]_{\hat{\g}}\\
&=&[s(D(x)), u]_{\hat{\g}}+\varrho(x)K(u)+[s(D(x)), K(u)]_{\hat{\g}}\\
&=&\varrho(D(x))u+\varrho(x)K(u)+\varrho(D(x))K(u),
\end{eqnarray*}
which implies that $\varrho:\g\lon\gl(\h)$ is a representation of the difference Lie algebra $(\g, D)$ on the vector space $\h$ with respect to the linear map $K$.

Let $s'$ be another section, and  $\varrho'$ be the corresponding representation of the difference Lie algebra $(\g,D)$ on $\h$ with respect to the linear map $K$. Since $s'(x)-s(x)\in\h$ and $\h$ is abelian, we have
\begin{equation*}
(\varrho'(x)-\varrho(x))u=[s'(x)-s(x), u]_{\hat{\g}}=0.
\end{equation*}
Thus the representation $\varrho$ is independent on the choice of sections.
\end{proof}

Let $s$ be a section. We further define
$
\omega\in\Hom(\wedge^{2}\g,\h),~ \chi\in\Hom(\g,\h)
$
by
\begin{eqnarray*}
\omega(x,y)&=&[s(x),s(y)]_{\hat{\g}}-s([x,y]_\g),\\
\chi(x)&=&\hat{D}(s(x))-s(D(x)).
\end{eqnarray*}
Define $S:\g\oplus \h\lon\hat{\g}$ by
$$
S(x+u)=s(x)+u.
$$
It is obvious that $S$ is an isomorphism between vector spaces. Transfer the difference Lie algebra structure on $\hat{\g}$ to $\g\oplus\h$ via the isomorphism $S$, we obtain a difference Lie algebra $(\g\oplus\h, [\cdot,\cdot]_{\omega}, D_{\chi})$, where $[\cdot,\cdot]_\omega$ and $D_{\chi}$ are given by
\begin{eqnarray*}
[x+u, y+v]_\omega&=&S^{-1}[S(x+u), S(y+v)]_{\hat{\g}}=[x, y]_\g+\varrho(x)v-\varrho(y)u+\omega(x,y),\\
D_{\chi}(x+u)&=&S^{-1}\hat{D}S(x+u)=D(x)+K(u)+\chi(x).
\end{eqnarray*}

\begin{thm}\label{thm:cohomologicalclass}
  With the above notations, $(\omega,\chi)$ is a $2$-cocycle of the difference Lie algebra $(\g, D)$  with coefficients in  $(\h, \varrho, K)$. Moreover, its cohomological class does not depend on the choice of sections.
\end{thm}
\begin{proof}
First by the fact that $[\cdot,\cdot]_{\omega}$ satisfies the Jacobi identity, we deduce that $\omega$ is a $2$-cocycle of the Lie algebra $(\g,[\cdot,\cdot]_\g)$ with coefficients in $(\h,\varrho)$, i.e. $\dM^{\CE}_{\varrho}\omega=0$.

Moreover, we have
 \begin{equation*}
[D_{\chi}(x+u),D_{\chi}(y+v)]_\omega
=D([x,y]_\g)+K(\varrho(x)v)-K(\varrho(y)u)+K(\omega(x,y))+\chi([x,y]_\g),
 \end{equation*}
and
\begin{eqnarray*}
&&[D_{\chi}(x+u),y+v]_{\omega}+[x+u, D_{\chi}(y+v)]_{\omega}+[D_{\chi}(x+u), D_{\chi}(y+v)]_{\omega}\\
&=&[D(x),y]_\g+\varrho(D(x))v-\varrho(y)K(u)-\varrho(y)\chi(x)+\omega(D(x),y)\\
&&+[x,D(y)]_\g-\varrho(D(y))u+\varrho(x)K(v)+\varrho(x)\chi(y)+\omega(x,D(y))\\
&&+[D(x),D(y)]_\g+\varrho(D(x))K(v)+\varrho(D(x))\chi(y)-\varrho(D(y))K(u)\\
&&-\varrho(D(y))\chi(x)+\omega(D(x),D(y)).
\end{eqnarray*}
By the fact that $D_{\chi}$ is a difference operator and
\begin{equation*}
K(\varrho(x)u)=\varrho(D(x))u+\varrho(x)K(u)+\varrho(D(x))K(u),
\end{equation*}
we have
\begin{eqnarray*}
0&=&-K(\omega(x, y))-\chi([x, y]_\g)+\omega(D(x,  y))+\omega(x, D(y))-\varrho(y)\chi(x)+\varrho(x)\chi(y)\\
&&+\varrho(D(x))\chi(y)-\varrho(D(y))\chi(x)+\omega(D(x), D(y))\\
&=&\partial\chi(x, y)+T(\omega)(x, y),
\end{eqnarray*}
which implies that $ T(\omega)+\partial\chi=0$. Therefore, $\delta_{\varrho}(\omega, \chi)=0,$ i.e. $(\omega,\chi)$ is a 2-cocycle.

Let $s'$ be another section and $(\omega',\chi')$ be the associated  2-cocycle. Assume that $s'=s+N$ for $N\in\Hom(\g,\h)$. Then we have
\begin{eqnarray*}
  (\omega'-\omega)(x,y)&=&[s'(x), s'(y)]_{\hat{\g}}-s'[x, y]_\g-[s(x), s(y)]_{\hat{\g}}+s[x, y]_\g\\
  &=&\varrho(x)N(y)-\varrho(y)N(x)-N([x, y]_\g)=\dM^{\CE}_{\varrho}N(x, y),\\
  (\chi'-\chi)(x)&=&\hat{D}(s'(x))-s'(D(x))-\hat{D}(s(x))+s(D(x))\\
  &=&K(N(x))-N(D(x)),
\end{eqnarray*}
which implies that $(\omega',\chi')-(\omega, \chi)=\delta_{\varrho}(N)$. Thus, $(\omega', \chi')$ and $(\omega, \chi)$ are in the same cohomology class.
\end{proof}

Isomorphisms between abelian extensions can be obviously defined as follows.
\begin{defi}
Let $(\hat{\g} ,\hat{D})$ and $(\tilde{\g}, \tilde{D})$ be two abelian extensions of a difference Lie algebra $(\g, D)$ by $(\h, K)$. They are said to be {\bf isomorphic} if there exists an isomorphism $\kappa:\tilde{\g}\lon\hat{\g}$ of difference Lie algebras such that the following diagram commutes:
  \begin{equation*}
\xymatrix@!0{0\ar@{->} [rr]&& \h \ar@{->} [rr] \ar'[d] [dd] \ar@{=} [rd] && \tilde{\g}\ar'[d] [dd]\ar@ {->} [rr] \ar@{->} [rd]^{\kappa}&& \g\ar@{=} [rd]\ar'[d] [dd]\ar@{->} [rr]&&0&\\
&0\ar@{->} [rr]&& \h\ar@{->} [rr]\ar@{->} [dd]&&\hat{\g}\ar@{->} [dd]\ar@{->} [rr]&&\g\ar@ {->} [dd]\ar@{->} [rr]&&0\\
0\ar@{->} [rr]&&\h\ar'[r] [rr] \ar@{=} [rd]&&\tilde{\g}\ar@{->} [rd]^{\kappa}\ar'[r] [rr]&&\g\ar@{=} [rd] \ar'[r] [rr]&&0&\\
&0\ar@{->} [rr]&& \h\ar@{->} [rr]&&\hat{\g}\ar@{->} [rr]&&\g\ar@{->} [rr]&&0.}
\end{equation*}
\end{defi}

\begin{thm}\label{extension 2}
   For a given representation $(\h, \varrho, K)$ of a difference Lie algebra $(\g, D)$, abelian extensions of $(\g, D)$  by   $(\h, K)$ are classified by the second cohomology group $\huaH^2(\g,D; \h,\varrho, K)$.
\end{thm}
\begin{proof}
 Let $(\hat{\g}, \hat{D})$ and $(\tilde{\g}, \tilde{D})$ be two isomorphic abelian extensions. Assume that $s$ is a section of $(\tilde{\g}, \tilde{D})$, and $(\tilde{\omega}, \tilde{\chi})$ is the corresponding 2-cocycle.
 Define $s'$ by
 \begin{equation*}
 s'=\kappa\circ s.
 \end{equation*}
 Then, it is obvious that $s'$ is a section of $(\hat{\g}, \hat{D})$. We denote by $(\hat{\omega}, \hat{\chi})$ the corresponding 2-cocycle. Then we have
 \begin{eqnarray*}
 \hat{\omega}(x,y)&=&[s'(x), s'(y)]_{\tilde{\g}}-s'([x,y]_\g)\\
 &=&[\kappa(s(x)),\kappa(s(y))]_{\tilde{\g}}-\kappa(s[x,y]_\g)\\
 &=&\kappa([s(x), s(y)]_{\hat{\g}}- s[x,y]_\g)\\
 &=&\tilde{\omega}(x,y).
 \end{eqnarray*}
 Similarly, we have $ \tilde{\chi}=\hat{\chi}.$
 By Theorem \ref{thm:cohomologicalclass},  isomorphic abelian extensions give rise to the same cohomological class in $\huaH^2(\g, D; \h, \varrho, K)$.

 For the converse part, we choose a 2-cocycle $(\omega, \chi)$, and define a bracket on $\g\oplus\h$ by
 \begin{equation*}
 [(x,u),(y,v)]_{\omega}=[x,y]_\g+\varrho(x)v-\varrho(y)u+\omega(x,y),\quad \forall x,y\in\g, u,v\in\h.
 \end{equation*}
 By $\dM^{\CE}_{\varrho}\omega=0$, it is straightforward to deduce that $(\g\oplus\h,[\cdot,\cdot]_{\omega})$ is a Lie algebra.
Define a linear map $D_{\chi}:\g\oplus \h\rightarrow\g\oplus\h$ by
 \begin{equation*}
 D_{\chi}(x, u)=D(x)+K(u)+\chi(x),\quad\forall x\in\g, u\in\h.
 \end{equation*}
 Since $\partial\chi+T(\omega)=0$, it is straightforward to deduce that $D_{\chi}$ is a difference operator. Thus $(\g\oplus\h, [\cdot,\cdot]_{\omega}, D_{\chi})$ is a difference Lie algebra, which is an abelian extension of $(\g, D)$ by $(\h, K)$.

Choose another 2-cocycle $(\omega', \chi')$, such that $(\omega, \chi)$ and $(\omega', \chi')$ are in the same cohomology class, i.e.
\begin{equation*}
(\omega-\omega', \chi-\chi')=(\dM^{\CE}_{\varrho}N, T(N)),
\end{equation*}
where $N\in\Hom(\g,\h)$, and denote the corresponding difference Lie algebra by $(\g\oplus\h, [\cdot, \cdot]_{\omega'}, D_{\chi'})$.
Define the linear map $\kappa:\g\oplus\h\rightarrow\g\oplus\h$ by
\begin{equation*}
\kappa(x, u)=(x, N(x)+u),\\
\end{equation*}
for all $x\in\g, u\in\h$. By $\omega-\omega'=\dM^{\CE}_{\varrho}N$, we deduce that $\kappa$ is a Lie algebra isomorphism.
Since
$
\chi(x)-\chi'(x)=T(N)x=-N(T(x))+K(N(x)),
$
we have $D_{\chi'}(\kappa(x, u))=\kappa(D_{\chi}(x, u))$. Then it is straightforward to deduce that the two abelian extensions are isomorphic.
\end{proof}

\section{Integrations of relative difference Lie algebras}\label{sec:int}
In this section, we show that any relative difference Lie algebra $(\g, \h, \rho, D)$  can be integrated to a relative difference Lie group $(G, H, \Phi, \huaD)$. We also extend the integration to the level of homomorphisms.
\begin{defi}{\rm(\cite{GLS})}
Let $\Phi: G\lon\Aut(H)$ be an action of $G$ on $H$. A smooth map $\huaD: G\lon H$ is called a relative difference operator on the Lie group $(G, e_G, \cdot_G)$ with respect to the action $\Phi$ if
\begin{equation*}
\huaD(a\cdot_G b)=\huaD(a)\cdot_H\Phi(a)\huaD(b), \quad \forall a, b\in G.
\end{equation*}

A {\bf relative difference Lie group}, denoted by $(G, H,\Phi, \huaD)$, consists of Lie groups $G, H$, an action  $\Phi: G\lon\Aut(H)$, and a relative  difference operator $\huaD$.

A relative difference operator from $G$ to $G$ with respect to the adjoint action $\Ad$ is called a {\bf  difference operator} on the Lie group $G$. A Lie group $G$ equipped with a difference operator $\huaD$ is called a {\bf difference Lie group}, and denoted by $(G, \huaD)$.
\end{defi}

\begin{rmk}
If the action $\Phi$ of $G$ on $H$ is trivial, then a relative difference operator from $G$ to $H$ is  a Lie group homomorphism.
\end{rmk}

\begin{defi}
Let $(G, H,\Phi,\huaD)$ and $(G', H',\Phi',\huaD')$ be two relative difference Lie groups. A homomorphism from $(G, H,\Phi,\huaD)$ to $(G',H',\Phi',\huaD')$ consists of a Lie group homomorphism $\Psi_G :G\rightarrow G'$ and a Lie group homomorphism $\Psi_H:H\lon H'$ such that
\begin{eqnarray}
\label{hom-rboG1}\huaD'\circ\Psi_G&=&\Psi_H\circ \huaD,\\
\label{hom-rboG2}\Psi_H\Phi(g)(h)&=&\Phi'(\Psi_G(g))(\Psi_H(h)),\quad \forall g\in G, h\in H.
\end{eqnarray}
\end{defi}

We can use the graphs of maps to characterize relative difference operators on Lie groups. The following proposition is obvious.
\begin{pro}
Let $\Phi: G\lon\Aut(H)$ be an action of $G$ on $H$. Then a smooth map $\huaD: G\lon H$ is a relative difference operator if and only if $Gr(\huaD)=\{(g, \huaD(g))|g\in G\}$ is a Lie subgroup of $G\ltimes_\Phi H$, where $G\ltimes_\Phi H$ is the semi-direct production Lie group with the multiplication $\cdot_\Phi$ given by
\begin{equation*}
(g_1,h_1)\cdot_{\Phi}(g_2,h_2)=(g_1\cdot_Gg_2, h_1\cdot_H\Phi(g_1)h_2), \quad \forall g_i\in G, h_i\in H, i=1,2.
\end{equation*}
\end{pro}

Let $(G, e_G, \cdot_G)$ and $(H, e_H, \cdot_H)$ be Lie groups whose Lie algebras are $\g$ and $\h$. Denote by $\exp_G$ and $\exp_H$ the exponential maps for the Lie groups $(G, e_G, \cdot_G)$ and $(H, e_H, \cdot_H)$ respectively.
Let $\Phi: G\rightarrow \Aut(H)$ be an action of $G$ on $H$.   Since
$\Phi(g)\in\Aut(H)$ for all $g\in G$, then $\Phi(g)_{*e_H}:
\h\rightarrow\h$ is a Lie algebra isomorphism. By $\Phi(g_1\cdot_G
g_2)=\Phi(g_1)\Phi(g_2)$, we have $\Phi(g_1\cdot_G
g_2)_{*e_H}=\Phi(g_1)_{*e_H}\Phi(g_2)_{*e_H}$. Thus we obtain a Lie
group homomorphism from $G$ to $\Aut(\h)$, which we denote by
$\tilde{\Phi}:G\lon\Aut(\h)$.
Then taking the differentiation, we obtain a Lie algebra homomorphism
$\rho:=\tilde{\Phi}_{*e_G}$ from the Lie algebra $\g$ to $\Der(\h)$. We call
$\rho$ the differentiated action of $\Phi$. In fact, Lie II theorem tells us that
$\Aut(H)\cong \Aut(\h)$, if $H$ is connected and simply
connected. Therefore when both $G$ and $H$
are connected and simply connected, given a Lie algebra action $\rho: \g \to
\Der(\h)$, there is a unique integrated action $\Phi: G\to \Aut(H)$
whose differentiation is $\rho$.  This procedure can be well explained by the following diagram:
\begin{equation}\label{eq:relation}
\small{ \xymatrix{G\ar@{=}[d]  \ar[rr]^{\Phi} & & \Aut(H) \ar[d]\\
 G \ar[rr]^{\tilde{\Phi}}\ar[d]^{\text{differentiation}} &  &  \Aut(\h) \ar[d]^{\text{differentiation}}  \\
 \g \ar[rr]^{  \rho=\tilde{\Phi}_{*e_G} } & &\Der(\h).}
}
\end{equation}

In \cite{GLST}, the authors proved that the differentiation of a relative difference operator $\huaD$ on a Lie group $G$ with respect to an action $(H, \Phi)$ is a relative difference operator $D$ on the Lie algebra $\g$ with respect to the action $(\h, \rho)$. Next, we give the integration of a relative difference Lie algebra $(\g, \h, \rho, D)$.

\begin{thm} \label{thm:obj}
Let $(\g, \h, \rho, D)$ be a relative difference Lie algebra. Let $G$ and $H$ be the connected and simply connected Lie
groups integrating $\g$ and $\h$ respectively, and $\Phi: G\to \Aut(H)$ be the
integrated action.
Then there is a relative difference operator $\huaD:G\lon H$ integrating the relative difference operator $D:\g\lon\h$.
\end{thm}
\begin{proof}
We consider the semi-direct product Lie group $G\ltimes_\Phi H$.
Its Lie algebra is the semi-direct product Lie algebra $\g\ltimes_\rho
\h$, and by Proposition \ref{graphch},  the graph of $Gr(D)$  is a Lie subalgebra
of $\g\ltimes_\rho \h$. Thus there exists a connected Lie subgroup $E$ of
$G\ltimes_\Phi H$ such that its Lie algebra is $Gr(D)$.
Define a smooth map $P_G: E\lon G$ by
\begin{equation*}
P_G(g,h)=g,\quad \forall (g,h)\in E.
\end{equation*}
Then $P_G: E \lon G$ is a Lie group homomorphism and its tangent map at the identity ${P_G}_*: Gr(D)\lon\g$ is given by
\begin{equation*}
{P_G}_{*}(x, D(x))=x,\quad \forall x\in\g,
\end{equation*}
which is an isomorphism from the Lie algebra $Gr(D)$ to the Lie algebra $\g$.  Thus $P_G: E\lon G$ is a Lie group isomorphism, which implies that there is a smooth map $\huaD: G\lon H$, such that $E\cong Gr(\huaD)$.

For all $g_1,g_2\in G$, we have
\begin{equation*}
(g_1, \huaD(g_1))\cdot_\Phi (g_2, \huaD(g_2))=(g_1\cdot_G g_2, \huaD(g_1)\cdot_H\Phi(g_1)\huaD(g_2))\in Gr({\huaD}),
\end{equation*}
which implies that
$\huaD(g_1\cdot_G g_2)=\huaD(g_1)\cdot_F\Phi(g_1)\huaD(g_2)$. Therefore,
$\huaD:G\lon H$ is a difference operator. Furthermore, since $Gr(\huaD)\cong E$, and the Lie algebra
of $E$ is $Gr(D)$, it follows that $\huaD_*=D$, and $\huaD$ is an
integration of $D$.

Now we give an explicit
formula of $\huaD$. Denote by $\EXP$ the exponential map for the Lie group
$G\ltimes_\Phi H$, and by $P_H$ the projection $G\ltimes_{\Phi}
H\to H$. For all $x\in\g, u\in\h$, it is obvious that
\begin{equation*}
\EXP (x,u)=(\exp_Gx, P_H\EXP(x,u)).
\end{equation*}
Since the Lie algebra of the Lie subgroup $E$  is $Gr(D)$, it follows that locally $\EXP(x,D(x))\in Gr(\huaD)\cong E$. Therefore,
\begin{equation} \label{eq:exp}
\huaD(\exp_Gx)=P_H\EXP (x,D(x)).
\end{equation}
Obviously the formula of $\huaD$ happens locally near the identity, the global formula follows from the fact that $G$ is connected and it can be written as products of elements near the
identity.
\end{proof}

Let $G, G', H$ and $H'$ be connected Lie groups whose Lie algebras are  $\g, \g', \h$ and $\h'$ respectively. Let $\Phi: G\rightarrow\Aut(H) $ and $\Phi': G'\lon\Aut(H')$ be actions of
  $G$ on  $H$ and $G'$ on $H'$ respectively. Let $\rho:\g\lon\Der(\h)$ and $\rho':\g'\lon\Der(\h')$ be the induced  actions of $\g$ on $\h$ and $\g'$ on $\h'$ respectively. Let  $\Psi_G:G\lon G'$ and $\Psi_H: H\lon H'$ be Lie group homomorphisms and $\psi_\g:\g\lon\g'$ and $\psi_\h:\h\lon\h'$ be the induced Lie algebra homomorphisms.

Let $(\Psi_G, \Psi_H)$ be a homomorphism from the relative difference Lie group $(G, H,\Phi,\huaD)$ to the relative difference Lie group $(G',H',\Phi',\huaD')$. It is obvious that $((\Psi_G)_{*e_G}, (\Psi_H)_{*e_H})$ is a homomorphism from the relative difference Lie algebra $(\g, \h, \rho, D)$ to the relative difference Lie algebra $(\g', \h', \rho', D')$. Now we extend the integration to the level of homomorphisms.
\begin{thm} \label{thm:int-functor}
 Let $(\g, \h,\rho, D)$ and $(\g', \h',\rho', D')$ be two relative difference Lie algebras. Let $G, G', H$ and $H'$ be connected and
  simply connected Lie groups integrating $\g, \g', \h$ and $\h'$ respectively.
  Let $\huaD: G\to H$ and $\huaD':
  G'\to H'$ be the integrated relative difference operators of $D$
  and $D'$ respectively. Let $(\psi_\g, \psi_\h)$ be a homomorphism from $(\g, \h,\rho, D)$ to $(\g', \h',\rho', D')$ and $\Psi_G, \Psi_H$ be Lie group homomorphisms integrating
Lie algebra homomorphism $\psi_\g$ and $\psi_\h$ respectively.  Then
$(\Psi_G,\Psi_H)$ is a homomorphism from the relative difference Lie group $(G, H, \Phi, \huaD)$ to the relative difference Lie group $(G', H', \Phi', \huaD')$.
\end{thm}
\begin{proof}
Suppose that $(\psi_\g,\psi_\h)$ is a homomorphism from $(\g, \h, \rho, D)$ to $(\g', \h', \rho', D')$.  It follows that $(\psi_\g,\psi_\h)$ is a Lie algebra homomorphism from the semi-direct
product Lie algebra $\g\ltimes_{\rho}\h$ to $\g'\ltimes_{\rho'}\h'$ and
$(\psi_\g,\psi_\h)(Gr(D))\subseteq Gr(D')$. Therefore, there exist  unique Lie group homomorphisms
$\Psi_G:G\lon G'$ and $ \Psi_H: H\lon H'$ such that $(\Psi_G, \Psi_H)$
is a Lie group homomorphism from the semi-direct product  Lie group
$G\ltimes_\Phi H$ to $G'\ltimes_{\Phi'} H'$. Since the image $(\Psi_G,\Psi_H)Gr(\huaD)$ is a Lie subgroup
of $G'\ltimes_{\Phi'} H'$ and its Lie algebra is $(\psi_\g,\psi_\h)(Gr(D))
\subseteq Gr(D')$, it follows that $(\Psi_G,\Psi_H)Gr(\huaD)\subseteq
Gr(\huaD')$ near the identity,  the global statement follows again
using the fact that both $(\Psi_G,\Psi_H)Gr(\huaD)$ and $Gr(\huaD)$ are
connected and they can be written as products of elements near the
identity.  This in turn implies that $\Psi_H\circ \huaD=\huaD'\circ \Psi_G$.
\end{proof}

\vspace{2mm}
\noindent
{\bf Acknowledgements. } This research is supported by NSFC (11922110).

 \end{document}